\documentclass[oneside,12pt]{paper}
\usepackage{graphicx,amsthm}
\usepackage[shortlabels]{enumitem}
\usepackage{amscd}
\usepackage{amssymb}
\usepackage{amsmath}
\usepackage{url}
\usepackage[dvipsnames]{xcolor}
\usepackage{pinlabel}
\usepackage{hyperref}

\usepackage[all,pdftex,arc,curve,color,frame]{xy}
\usepackage{xy}
\xyoption{all}

\usepackage[margin=3cm, marginpar=2.5cm]{geometry}

\theoremstyle{theorem}
\newtheorem{thm}{Theorem}[section]

\newtheorem{prop}[thm]{Proposition}
\newtheorem{cor}[thm]{Corollary}

\newtheorem{quest}[thm]{Question}

\theoremstyle{remark}
\newtheorem{defn}[thm]{Definition}
\newtheorem{rmk}[thm]{Remark}
\newtheorem{exa}[thm]{Example}
\newtheorem{exe}[thm]{Exercise}

\newcommand{\C}{\mathbb{C}}

\newcommand{\Z}{\mathbb{Z}}

\newcommand{\F}{\mathbb{F}}

\newcommand{\CP}{\mathbb{CP}}
\renewcommand{\P}{\mathbb{P}}
\newcommand{\CPbar}{\overline{\CP}{\vphantom\CP}^2}

\renewcommand{\epsilon}{\varepsilon}

\newcommand{\co}{\colon}
\newcommand{\inv}{^{-1}}

\newcommand{\del}{\partial}

\newcommand\sd{\mkern1.5mu{:}\mkern1.5mu}

\newcommand{\mfig}[1]{\includegraphics[width=0.5\textwidth]{figures/#1}}
\newcommand{\mlfig}[1]{\includegraphics[width=0.6\textwidth]{figures/#1}}
\newcommand{\bfig}[1]{\includegraphics[width=0.9\textwidth]{figures/#1}}

\DeclareMathOperator{\Sing}{Sing}
\DeclareMathOperator{\Cone}{Cone}

\DeclareMathOperator{\id}{id}

\DeclareMathOperator{\lk}{lk}

\DeclareMathOperator{\Aut}{Aut}

\title{Surfaces in 4-manifolds and complex curves}
\author{Marco Golla}

\begin{document}

\maketitle

\begin{abstract}
These are lecture notes from a mini-course taught at Winterbraids XIII (Montpellier, 2024). The main character of these notes are curves in the complex projective plane, $\CP^2$, viewed from a topological perspective.
\end{abstract}

\section{Introduction}

Plane curves are among the oldest objects in mathematics, which have been studied since the ancient Greeks. Complex curves have been discovered a lot more recently, but they have been studied at least since Riemann's work on surfaces.

In these notes, I will discuss curves in the complex projective planes and their singularities from the viewpoint of a low-dimensional topologist. This perspective dates back to the work of Zariski and van Kampen on fundamental groups~\cite{Zariski,vanKampen} and to the Italian school of algebraic geometry~\cite{Chisini}. More recent contributions are due to Libgober~\cite{LibgoberA, Libgober-homotopy}, Moishezon~\cite{Moishezon-noniso}, Teicher~\cite{MoishezonTeicher-wiringI, MoishezonTeicher-wiringII}, Artal Bartolo, Carmona Ruber, Cogolludo Agust\'in~\cite{ArtalCarmonaCogolludo, ACI}, Florens~\cite{CogolludoFlorens}, and many more.

Here, a \emph{curve} will be the zero-set $C$ of a homogeneous (and \emph{reduced}, or square-free) polynomial in three variables, viewed as a subset of the complex projective plane, $\CP^2$. We will be interested in the topology of $C$, of $\CP^2\setminus C$, and of the pair $(\CP^2,C)$. There are mainly three questions I am interested in, and the aim of these notes is to give a brief introduction to all three of them.

\begin{quest}[The geography problem]
Which \emph{combinatorial types} of curves can be realised as complex curves?
\end{quest}

Here by combinatorial type we essentially mean the discrete data that a curve has: number of components, types of singularities, perhaps the incidence relations. This will be made more explicit in the next section.

Geography problems are very common in algebraic geometry: for instance, what are the possible points in the $\Z^2$-plane realised as $(c_1^2(S),c_2(S))$ for some compact complex surface? Or even, what are the possible Hodge numbers of non-singular complex algebraic varieties? The question stated above is a relative version of these geography problems.

\begin{quest}[The botany problem]
Given a combinatorial type, in how many \emph{different} ways can it be realised?
\end{quest}

Here the key to the question is defining what we mean by ``different'', which we will do in the next section. There are several equivalence relations one can look at, and we will focus on a very topological one in these notes. The keyword here is \emph{Zariski pairs}, which are pairs of objects that have the same combinatorial type but different topology (whatever this means, for the moment).

\begin{quest}[Algebraic geometry restricts topology]
What is the difference between the set of topological objects that look algebro-geometric and the set of algebro-geometric objects?
\end{quest}

This question is extremely vague, but let me try to give an analogy first.

Every non-singular complex projective surfaces is closed, smooth, oriented 4-manifold (see, for instance,~\cite[Theorem 1.2.33]{GompfStipsicz} for a reference).
What can we say about their topology?
For instance, every finitely-presented group is the fundamental group of a closed, smooth, oriented 4-manifolds.
Can non-singular complex projective surfaces also have arbitrary (finitely presented) fundamental groups?
The answer is no, for instance because the first Betti number of a K\"ahler manifold (of which a complex projective surfaces is a special case) is even (see~\cite[Section I.13]{BHPV}): for example, $\Z$ cannot be the fundamental group of a K\"ahler surface\footnote{Note that, however, $\Z$ is the fundamental group of a compact complex surface! The Hopf surface $\C^2\setminus \{0\}/z \sim 2z$ is diffeomorphic to $S^1\times S^3$.}.
In this sense, being an algebro-geometric object (a complex projective surface) restricts the underlying topology (the fundamental group of the underlying 4-manifold).

Back to curves, we can ask which groups appear as fundamental groups of complements of a curve in $\CP^2$, and ask whether, for every $G$ which is the fundamental group of the complement of a surface $S \subset \CP^2$ with conical singularity, there also exist a complex curve whose complement has fundamental group $G$.
We will not discuss this question in detail, but only touch upon it on occasions.

The main results we prove here are three, all classical. The subdivision of topics roughly reflects the subdivision of the mini-course into three lectures.

\begin{itemize}
\item We compute of the homology of $\CP^2\setminus C$, which turns out to be a rather coarse invariant of $C$ (Theorem~\ref{t:homology}). We also discuss some results about the fundamental group of $\CP^2\setminus C$, without going into a lot of details.

\item We determine the homeomorphism type of $C$, by proving the adjunction formula (Theorem~\ref{t:adjunction}). We also give a singular version of the same formula (Theorem~\ref{t:singularadjunction}). This is also a good excuse to introduce branched covers between Riemann surfaces and Milnor fibres of plane curve singularities.

\item We exclude some configurations of singularities from appearing on a curve of some degree, using branched covers of Riemann surfaces and of 4-manifolds. These are rather flexible proofs: some of them apply to pseudo-holomorphic curves (which we will not discuss in detail), while others apply to more general (real) surfaces with conical singularities.
\end{itemize}

\paragraph{Organisation} In Section~\ref{s:curves} we define curves and their singularities more formally, we state some fundamental properties they have, and we formulate more properly the botany and geography questions. In Section~\ref{s:complements} we look at the algebraic topology of curve complements, focusing on the homology and the fundamental group. In Section~\ref{s:adjunction} we show the degree-genus formula, which expresses the genus of a non-singular complex curves in terms of the degree, and we give some proofs of the adjunction formula from which this computation follows. In Section~\ref{s:singularadjunction} we define Milnor numbers of singularities and prove a singular version of the adjunction formula. In Section~\ref{s:blow-ups} we discuss blow-ups and their interactions with curves. In Section~\ref{s:Fano} we use blow-ups and branched covers of 4-manifolds to obstruct the existence of the Fano configuration. Finally, in Section~\ref{s:RiemannHurwitz} we go back to branched covers of Riemann surfaces and use them to obstruct the existence of certain configurations on singularities on curves, mostly through examples.

\paragraph{Acknoweldgements} I would like to thank the organisers of the conference Winterbraids XIII for giving me the opportunity to teach a mini-course on this topic and to write these lecture notes. I would also like to thank the participants who motivated me with good questions and comments during the conference. Finally, I would like to thank Paul Brisson, Peter Feller, and Vincent Florens for many relevant conversations.


\section{Complex curves and surfaces with conical singularities}\label{s:curves}

We begin by recalling some basic definitions about projective spaces and projective curves.

\begin{defn}
The \emph{complex projective $n$-space} is the manifold $(\C^{n+1}\setminus\{0\})/\C^*$, where $\C^*$ acts diagonally by multiplication on $\C^{n+1}$: $\lambda\cdot(z_0,\dots,z_n) = (\lambda z_0, \dots, \lambda z_n)$.
\end{defn}

What is not obvious from the definition is that $\CP^n$ is a manifold. In fact, it is even a non-singular complex variety, as one can check in each chart of the form $\{z_k \neq 0\}$ by constructing explicit slices for the action and computing the transition maps.

A point in $\CP^n$ is thus an the equivalence class, and we denote the point corresponding to $(z_0, \dots, z_n) \in \C^{n+1}$ by $(z_0\sd\dots \sd z_n)$. When $n = 1$ we have the complex projective line, also called the \emph{Riemann sphere}, which is diffeomorphic to the 2-sphere $S^2$. When $n=2$ we have the complex projective plane. In these cases we use $(x\sd y)$ and $(x\sd y \sd z)$ as coordinates, to make the notation lighter.

Recall that $H^*(\CP^n)$, endowed with the cup product, is isomorphic as a ring to $\Z[h]/(h^{n+1})$: an explicit isomorphism sends $h$ to the hyperplane class (i.e. is Poincar\'e dual to the homology class of a hyperplane, which we still denote with $h$), and $h^n$ is the orientation class (induced by the complex structure). This is obvious for $n = 1$, once we know that $\CP^1$ is diffeomorphic to $S^2$. We will prove this below for the case $n = 2$.

\begin{defn}
A \emph{(reduced, plane, projective, complex) curve} $C = \{f = 0\}$ is the zero-set of a homogeneous and square-free three-variable polynomial $f \in \C[x,y,z]$. The \emph{degree} $\deg C$ of $C$ is defined as the degree of $f$. $C$ is called \emph{irreducible} if $f$ is irreducible. If $f = f_1\cdots f_c$ is a product of irreducible factors, then each curve $\{f_k = 0\}$ is called an \emph{irreducible component} of $C$.
\end{defn}

We will systematically drop most adjectives and simply talk about \emph{curves}, implicitly assuming that all curves we talk about are plane, projective, reduced, and defined over the complex numbers.

Irreducibility in algebraic geometry is the analogue of connectedness in topology, so we can think of the irreducible components of $C$ as its connected components. However, note that, as a topological space, a reducible curve is always connected (a fact which we will see shortly).

\begin{defn}
A point $p$ on a curve $C = \{f = 0\}$ is called \emph{singular} if $\nabla f(p) = 0$, and \emph{non-singular} otherwise. We denote with $\Sing(C)$ the set of singular points of $C$.
\end{defn}

By the implicit function theorem, locally around each of its non-singular points, $C$ looks like $\{y=0\} \subset \C^2$, meaning that there are a neighbourhood $U$ of $p$, a neighbourhood $V$ of the origin in $\C^2$, and a biholomorphism $\phi\co U \to V$ that induces a biholomorphism of triples $(U,C \cap U,p) \cong (V, \{y = 0\} \cap V, 0)$. Singular points are more complicated, as we will see below.

\begin{thm}\label{t:holomorphicimage}
Every irreducible curve is the image of a holomorphic map $u\co\Sigma \to \CP^2$, for some connected Riemann surface $(\Sigma, j)$. This map is unique up to biholomorphisms of $\Sigma$ if we require that it is somewhere injective.
\end{thm}

\begin{proof}[Sketch of proof]
Every curve singularity can be \emph{resolved} by \emph{blowing up} (more on this in Section~\ref{s:Fano}). This means that, given $C$, there are a complex surface $X$, a map $\pi$, and a non-singular, irreducible curve $\tilde C \subset X$ such that the restriction of $\pi$ to $X\setminus \pi\inv(\Sing(C))$ a biholomorphism onto $\CP^2 \setminus \Sing(C)$ and the restriction of $\pi$ to $\tilde C$ one-to-one away from $\pi\inv(\Sing(C))$.

Letting $\Sigma = \tilde C$ and $u = \pi|_\Sigma$ we obtain the statement.
\end{proof}

\begin{thm}[{\cite[Theorem~2.10]{Milnor-singularities}}]
Let $C$ be a curve and $p \in C$. Then there exists a link $L = L_{(C,p)} \subset S^3$ and a closed neighbourhood $U$ of $p$ such that $(U,C\cap U, p)$ is \emph{homeomorphic} to $(\Cone(S^3), \Cone(L), {\rm vertex})$.
\end{thm}

\begin{defn}
The link $L_{(C,p)}$ in the theorem above is called the \emph{link} of the singularity of $C$ at $p$.
\end{defn}

Note that the link of a singularity is more classically defined as the intersection of $C$ with a sphere of small radius centred at $p$, a definition which goes back at least to Brauner~\cite{Brauner} and Brieskorn~\cite{Brieskorn}, generalised to higher dimensions and codimensions. Motivated by the theorem above, we give a definition that generalises complex curves by viewing them as topological objects.

\begin{defn}
We say that two complex curve singularities are \emph{(topologically) equivalent} if their links are diffeomorphic.
\end{defn}

\begin{rmk}
Note that the concept of topological equivalence for singularities is much looser than that of diffeomorphism or biholomorphism. For instance, consider the family of singularities of $\{xy(x-y)(x-\lambda y) = 0\} \subset \C^2$ at the origin: topologically, these are all equivalent, since their link is always the union of four fibres of the Hopf fibration on $S^3$. However, the cross ratio gives an obstruction for two singularities in these family to be biholomorphic (why?). Similar examples can be concocted of non-diffeomorphic, but topologically equivalent, singularities.
\end{rmk}

\begin{exa}
Each of these examples can be taken as an exercise.
\begin{itemize}
\item If $p$ is a non-singular point of $C$, then the link $L_{(C,p)}$ is the unknot. The converse is also true, as we will mention below.
\item If $p < d$ and $C = \{x^p z^{d-p} - y^d = 0\}$, then the link of $C$ at the point $(0\sd 0\sd1)$ is the torus link $T(p,d)$. In particular, if $p = 2$ and $d=3$ the link is a trefoil knot.
\end{itemize}
\end{exa}

\begin{defn}
A \emph{surface with conical singularities} in a $4$-manifold $X$ is a compact subspace $F$ of $X$ that at each point is locally homeomorphic to the cone over a link. The \emph{singularities} or \emph{singular points} of $F$ are the points where the link is not the unknot. A singular point of $F$ is called a \emph{cusp} or \emph{cuspidal} if the link has one component.
\end{defn}

Note that the definition automatically implies that each such $F$ has finitely many singularities.

\begin{rmk}
We can make the definition more restrictive in at least two ways. For once, if $X$ is smooth we could require the curve to be \emph{smoothly} embedded away from its singular points (which in turn is equivalent to asking that the surface is isotopic to a PL-immersed surface with respect to some triangulation of $X$). If $X$ supports a symplectic form $\omega$, we can require that $F$ is $J$-holomorphic for some almost-complex structure compatible with $\omega$. We will not use these variants in these notes.
\end{rmk}

We want to define some equivalence relations on curves.

\begin{defn}
Two irreducible curves are \emph{(topologically) equisingular} if the collection of their singularities are the same. We are going to drop the adverb ``topologically'' throughout.

Two curves $C$, $C'$ are:
\begin{itemize}\itemsep 0pt
\item \emph{weakly combinatorially equivalent} if they have components which are equisingular and if they have the same collections of singularities;
\item \emph{strongly combinatorially equivalent} if there is an equisingular bijection of the components of $C$ and $C'$ that preserves the degrees of the components and the incidence relations;
\item \emph{topologically equivalent} if there is a homeomorphism of pairs $(\CP^2,C) \cong (\CP^2,C')$;
\item \emph{biholomorphic} if there is a projective transformation of $\CP^2$ mapping $C$ to $C'$.
\end{itemize}

The \emph{weak/strong combinatorial type} of a curve is its weak/strong combinatorial equivalence class (or, which amounts to the same, the data of its singularities and, in the case of strong equivalence, the incidence relations). The \emph{topological type} of a curve is the homeomorphism class of the corresponding pair.
\end{defn}

The four equivalence relations above are clearly in order from weaker to stronger, and finding examples of curves satisfying one but not the next is not easy in all cases. The strong combinatorial equivalence is a bit loosely defined (and we will not really use it), so let us give an example of what it means, in the special case of \emph{line arrangements}.

\begin{defn}
A \emph{line arrangement} is a complex curve whose irreducible components are lines, i.e. they have degree $1$.
\end{defn}

Line arrangements can only have non-singular components (since lines cannot be singular) and every two lines intersect transversely once.
The only singularities it can have are transverse multiple points, which are topologically equivalent to $\{x^m = y^m\} \subset \C^2$, whose link is the union of $m$ fibres of the Hopf fibration of $S^3$.
The quantity $m$ is called the \emph{multiplicity} of the singularity. In this case, the weak combinatorial type of the arrangement is encoded by the sequence $(t_2, t_3, \dots)$, where $t_m$ is the number of points in the arrangement of multiplicity $m$.
On the other hand, two line arrangements $L = L_1 \cup \dots \cup L_d$ and $L' = L'_1 \cup \dots \cup L'_d$ are strongly combinatorially equivalent if there exists a permutation $\sigma$ of $\{1,\dots,d\}$ such that whenever $L_{i_1}, \dots, L_{i_m}$ meet at a point, then $L'_{\sigma(i_1)}, \dots, L'_{\sigma(i_m)}$ also meet at a point.

\begin{exa}
In Figure~\ref{f:samewc} we show an example of two weakly combinatorially equivalent line arrangements that are not strongly combinatorially equivalent. They are two arrangements of degree 6, with two triple points and nine double points: on the left, the two triple points are on the same line, on the right they are not. (Alternatively: on the left there is a line without triple points, on the right there is no such line.)
\end{exa}

\begin{figure}[h]
\centering
\bfig{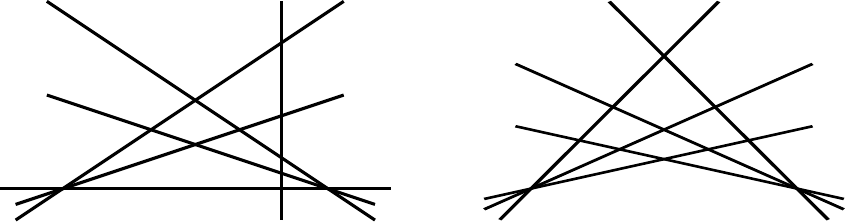}
\caption{Two arrangements that are weakly but not strongly combinatorially equivalent.}\label{f:samewc}
\end{figure}

We can now reformulate the three questions from the introduction in a more precise manner.

\begin{quest}[Geography]
Which weak/strong combinatorial types are realised by complex curves?
\end{quest}

\begin{quest}[Botany]
Given a weak/strong combinatorial type, how many topological types is it partitioned into? Into how many (moduli spaces of) biholomorphic equivalence classes?
\end{quest}

\begin{quest}[Algebraic geometry vs topology]
Which weak/strong combinatorial types are realised by surfaces with conical singularities or by pseudo-holomorphic curves, but not by complex curves?
\end{quest}


\section{Curve complements}\label{s:complements}

If two curves are biholomorphic or topologically equivalent, then in particular their \emph{complements} are going to be homeomorphic. Therefore, every homeomorphism (or homotopy equivalence) invariant of their complement is a topological invariant of the curve.

The goal of this section is to investigate the algebraic topology of these complements.

Let us start with a simple, and rather coarse invariant: the Euler characteristic. The Euler characteristic of the complement of a curve $C$ is easy to compute:
\[
\chi(\CP^2\setminus C) = \chi(\CP^2) - \chi(C) = 3- \chi(C).
\]
In turn, $\chi(C)$ only depends on the weak combinatorics of the curve. This is rather cumbersome to express in general, but  something can be said about line arrangements.

\begin{exe}
If $L$ is a line arrangement of degree $d$ with weak combinatorics $(t_2, t_3, \dots)$, then
\[
\chi(L) = 2d-\sum (m-1)t_m.
\]
\end{exe}

The topology of the complement is in fact somewhat simpler than one would expect in full generality.

\begin{thm}
If $C \subset \CP^2$ is a complex curve, then $\CP^2\setminus C$ has a handle decomposition with handles of indices at most $2$.
\end{thm}

\begin{proof}[Sketch of proof]
Using the Veronese embedding $\iota\co\CP^2 \hookrightarrow \CP^{{d+2 \choose 2}-1}$, one can express $C$ as the intersection of $\iota(\CP^2)$ with a hyperplane $H \subset \CP^{{d+2 \choose 2}-1}$. In particular, $\CP^2\setminus C$ is an affine non-singular variety, so it is a Stein manifold. Stein manifolds have handle decompositions with handles of index at most 2 (this is a classical fact about plurisubharmonic functions: for instance, see~\cite[Chapter~8]{OzbagciStipsicz} and~\cite[Section~7]{Milnor-Morse} for more details).
\end{proof}

\begin{rmk}
Libgober gave an explicit 2-dimensional cell complex onto which $\CP^2\setminus C$ retracts~\cite{Libgober-homotopy}, and very recently Sugawara gave an explicit handle decomposition~\cite{Sugawara}. Both are expressed in terms of braid monodromy.
\end{rmk}

\begin{thm}\label{t:homology}
If $C$ is a curve with components of degrees $d_1, \dots, d_c$, then
\[
H_1(\CP^2\setminus C) = \Z^{\oplus c}/(d_1, \dots, d_c)\Z.
\]
In particular, if $\gcd(d_1, \dots, d_c) = 1$, then $H_1(\CP^2\setminus C)$ is free of rank $c-1$. Moreover, $H_2(\CP^2\setminus C)$ is free and all higher homology groups vanish.
\end{thm}

Before diving into the proof, let us take a short detour.
The following fact was mentioned earlier, but let us give a proper statement and prove it.

\begin{prop}\label{p:curvehomology}
$H_2(\CP^2)$ is generated by the homology class $h$ of a complex line, and the intersection form on it the unique positive definite unimodular quadratic form of rank $1$. A curve of degree $d$ represents the homology class $dh$.
\end{prop}

\begin{proof}
Recall that $\CP^2$ is a smooth, closed, oriented 4-dimensional manifold, so it satisfies Poincar\'e duality over the integers. Recall also that $\CP^2 = S^5/S^1$, and from the long exact sequence of the fibration $S^1 \to S^5 \to \CP^2$ we obtain that $\CP^2$ is simply-connected and that $\pi_2(\CP^2) = \Z$. From the Hurewicz theorem and Poincar\'e duality, $H_*(\CP^2)$ is supported in degrees $0$, $2$, and $4$, and it has rank $1$ in each of these degrees\footnote{An alternative proof, which is more classical and more elementary, is to give an explicit handle cell decomposition of $\CP^2$ consisting of only one 0-cell, one 2-cell, and one 4-cell, where the 4-cell is attached to $S^2$ along the Hopf map. This is the proof in Hatcher's book~\cite[Example~0.6]{Hatcher}}.

We want to identify the generator of $H_2(\CP^2)$. Take two distinct complex lines $L$ and $L'$. Since they are orientable, they each represent a homology class $[L], [L'] \in H_2(\CP^2)$. Since one can deform the equation of $L$ onto that of $L'$, $[L] = [L'] =: h$ ($h$ for hyperplane).

Now, $L$ and $L'$ are two distinct complex lines, so they intersect transversely and positively at one point. In particular, their intersection product is $+1$. It follows that the homology class $h$ is primitive (why?) so it generates $H_2(\CP^2)$.

Now consider a curve $C$ of degree $d$. Choose a generic line $L$, so that the intersection between $L$ and $C$ is transverse. (That such a line exists, for instance, follows from Sard's theorem.) By the fundamental theorem of algebra, $L$ and $C$ intersect $d$ times. Since both $C$ and $L$ are complex curves, all intersections are positive, so $[C]\cdot h = d$ and therefore $[C] = dh$.
\end{proof}

As a corollary, we get a weak form of B\'ezout's theorem.

\begin{exe}[Weak B\'ezout's theorem]
If two curves $C$ and $C'$, of degrees $d$ and $d'$ respectively, meet transversely, then they meet at $dd'$ points. \emph{Assuming that each irreducible curve is connected}, show that, as mentioned above, each curve is connected.
\end{exe}

Let us get back to the proof of Theorem~\ref{t:homology}

\begin{proof}[Proof of Theorem~\ref{t:homology}]
We want to compute $H_k(\CP^2\setminus C)$ for each $k$. Call $N$ a regular neighbourhood of $C$, so that $N$ retracts onto $C$ and $\CP^2\setminus \overline{N} \cong \CP^2\setminus C$. By Poincar\'e--Lefschetz duality, excision, and homotopy invariance we have:
\[
H_k(\CP^2 \setminus N) \cong H^{4-k}(\CP^2 \setminus N, \del N) \cong H^{4-k}(\CP^2, N) \cong H^{4-k}(\CP^2, C).
\]
Consider the long exact sequence of the pair $(\CP^2,C)$ in cohomology. Call $c$ the number of components of $C$. Note that since $C$ is an oriented surface with conical singularities, its first homology is free and its second homology is free of rank $c$ (exercise!). By the universal coefficients theorem, the same holds in cohomology. From the exercise above, we also know that $C$ is connected.
\[
0 \to H^0(\CP^2,C) \to H^0(\CP^2) \to H^0(C) \cong \Z,
\]
so $H^0(\CP^2,C) = 0$ and we can continue with the rest of the long exact sequence.
\[
0 \to H^1(\CP^2,C) \to H^1(\CP^2) = 0,
\]
so $H^1(\CP^2,C) = 0$. We now get to the core of the argument:
\[
0 \to H^1(C) \to H^2(\CP^2,C) \to H^2(\CP^2) \to H^2(C) \to H^3(\CP^2,C) \to H^3(\CP^2) = 0.
\]
From Proposition~\ref{p:curvehomology}, the map $H_2(C) \to H_2(\CP^2)$ sends the fundamental class of a component $C_i$ of $C$ of degree $d_i$ onto $d_i h$. Therefore the map on cohomology, which is the transpose to the map in homology since we have no torsion, sends the generator of $H^2(\CP^2)$ to the vector $(d_1, \dots, d_c)$ in the basis of $H^2(C)$ given by the fundamental classes (in cohomology, so Poincar\'e duals to points) of the components.

Two things follow: $H^2(\CP^2) \to H^2(C)$ is injective, so $H^1(C) \cong H^2(\CP^2,C)$, and $H^3(\CP^2,C)$ is equal to the cokernel of the map $\Z \ni 1\mapsto (d_1, \dots, d_c) \in \Z^c$, from which the statement follows.
\end{proof}

A lot more information is contained in the fundamental group of a curve. However, computing a presentation of these fundamental groups is fairly hard. This was first done by van Kampen~\cite{vanKampen} and their ideas lead to the concept of braid monodromy~\cite{MoishezonTeicher-wiringI, MoishezonTeicher-wiringII, CohenSuciu, Cogolludo-braidmonodromy}. The computation of any non-trivial fundamental groups go beyond the scope of these lecture notes, but I would like to state some results and present some interesting examples (due to Zariski).

\begin{thm}[Folklore, see~\cite{Cogolludo-braidmonodromy}]\label{t:cyclicpi1}
If $C$ is a non-singular curve of degree $d$, then $\pi_1(\CP^2\setminus C)$ is cyclic of order $d$.
\end{thm}

We present a proof from~\cite[Lemma~1.1]{Libgober-homotopyII}\footnote{I learnt about this proof from Khash's beautiful \href{https://mathoverflow.net/questions/464026/computing-pi-1-of-the-complement-of-a-non-singular-plane-curve}{MathOverflow answer}.}.

\begin{proof}
Suppose that $C = \{F = 0\}$, where $F$ is a homogeneous polynomial of degree $d$.
Since $H_1(\CP^2\setminus C)$ is isomorphic to $\Z/d\Z$, as seen above, it suffices to show that $\pi_1(\CP^2\setminus C)$ is Abelian. To see that, recall that there is a fibration $S^1 \to S^5\setminus L \to \CP^2\setminus C$ obtained by restricting the fibration $S^1\to S^5\to \CP^2$. $L$ is the link of the singularity of the hypersurface of $\C^3$ defined by $\{F = 0\}$ (this hypersurface is known as the cone over $C$, in the algebro-geometric sense).
Since $C$ is non-singular, this singularity is isolated.

Milnor's theory of hypersurface singularities tells us that there is a fibration $M \to S^5\setminus L \to S^1$ (the Milnor fibration) and that $M$ (the Milnor fibre) is homotopy equivalent to a wedge of $2$-spheres. In particular, $M$ is simply-connected. From the long exact sequence in homotopy induced by this fibration we get:
\[
1 = \pi_1(M) \to \pi_1(S^5\setminus L) \to \pi_1(S^1) \to \pi_0(M) = 1,
\]
so $\pi_1(S^5\setminus L)$ is infinite cyclic.

From the long exact sequence of the fibration $S^1 \to S^5\setminus L \to \CP^2\setminus C$ we now obtain:
\[
\Z = \pi_1(S^1) \to \pi_1(S^5\setminus L) \to \pi_1(\CP^2\setminus C) \to \pi_0(S^1) = 1,
\]
so $\pi_1(\CP^2\setminus C)$ is Abelian.
\end{proof}

There is a generalisation of Theorem~\ref{t:cyclicpi1} for \emph{nodal curves}. Recall that a nodal curve is a curve with only ordinary double points as singularities; that is, all its singularities are locally modelled on $\{x^2 = y^2\}$.

\begin{thm}[Fulton, Deligne]\label{t:nodalcurves}
Suppose that $C$ is a nodal curve. Then $\pi_1(\CP^2\setminus C)$ is Abelian.
\end{thm}

In particular, Theorem~\ref{t:nodalcurves} implies that the fundamental group is a combinatorial invariant of the nodal curve $C$.
In fact, it only depends on the number of components of $C$ and the greatest common divisor of the degrees of its components.

The situation changes drastically when the singularity get even slightly more complicated, as the following theorem shows.

\begin{thm}[Zariski]
There exist two irreducible curves of degree $6$, $C_1$ and $C_2$, each having six simple cusps, i.e. six singularities whose link is a positive trefoil, such that $\pi_1(\CP^2\setminus C_1) \cong \Z/6\Z$ and $\pi_1(\CP^2\setminus C_2) \cong \Z/2\Z * \Z/3\Z$.
\end{thm}

An example of the curve $C_2$ is the one defined by the equation $27p^2 - 4q^3 = 0$, where $q$ is a degree-2 polynomial and $p$ is a degree-3 polynomial, such that the conic $\{q = 0\}$ and the cubic $\{p=0\}$ intersect transversely in six points. An explicit equation for $C_1$ was found by Oka~\cite{Oka-Zariskisextics} much later than Zariski's existence proof:
\[
27x^2 (x - z)^2 (x^2 + 2 xz - y^2 + z^2) + 9 (x^2 - z^2) (y^2 - z^2)^2 - (y^2 - z^2)^3
\]

Zariski's sextics are the first example of a \emph{Zariski pair}: two equisingular curves whose complements are not homeomorphic. Since then, more examples have been found, including examples of line arrangements: see the survey~\cite{Zariski-survey} by Artal Bartolo, Cogolludo Agust\'in and Tokunuga.

\section{Branched covers: dimension 2}\label{s:adjunction}

The goal of this section is to study the topology of the curve itself, without considering its embedding in $\CP^2$. The main result we present is the adjunction formula (for curves in $\CP^2$).

\begin{thm}[The adjunction formula]\label{t:adjunction}
If $C$ is a non-singular curve of degree $d$ in $\CP^2$, then $C$ is a compact orientable surface of genus $g = \frac12(d-1)(d-2)$.
\end{thm}

We will see in the next section the singular analogue of this theorem, which requires looking at bit more closely at singularities and Milnor fibres (who already made an appearance in the last section).

We will use the following fact from algebraic geometry.

\begin{thm}[Folklore]\label{t:allcurvesareiso}
Any two non-singular curves $C, C'$ of degree $d$ in $\CP^2$ are isotopic (and in particular they are homeomorphic).
\end{thm}

\begin{proof}[Rough sketch of proof]
The set of all curves of degree $d$ in $\CP^2$ (including non-reduced curves) is a projective space $P$ of dimension ${{d+2} \choose 2} - 1$. The set of singular curves is a (complex, singular) codimension-1 subvariety of $P$. A complex codimension-1 subvariety does not disconnect $P$, so the set of non-singular curves is connected. A path between the two curves now induces an isotopy, by Ehresmann's fibration theorem.
\end{proof}

Thanks to the theorem above, to prove the adjunction formula we just need to show it for one specific curve.
There are four ways to do this. We will give one detailed proof, and sketch the other three.

Before going into the proof, recall that a holomorphic map between Riemann surfaces is locally modelled on the map $\mu_n \co \C\to \C$, $z\mapsto z^n$. We can take this local model and use it as the definition of a branched cover (of surfaces).

\begin{defn}\label{d:bc2}
Let $X$ and $Y$ be two compact, oriented surfaces, and $p\co X\to Y$ a smooth map. $p$ is called a \emph{branched cover} if there exists a finite set $B \subset Y$ such that, calling $X' = X \setminus p\inv(B)$, $Y' = Y \setminus B$, and $p'\co X' \to Y'$ the restriction of $p$ to $X'$ (with the restriction on its codomain, too):
\begin{itemize}\itemsep 0pt
\item $p'$ is a covering map (i.e. a local homeomorphism);
\item for every point $x \in p\inv(B)$, there are complex coordinates around $x$ and $p(x)$ in which $p$ is $\mu_{n_x}$.
\end{itemize}
$B$ is called the \emph{branching set}, and $p'$ is the \emph{cover} associated to $p$. We also say, by a slight abuse of terminology, that $X$ is a branched cover of $Y$, or that $X$ is a cover of $Y$ branched over $B$. $n_x$ is called the \emph{index of ramification} of $p$ at $x$, which we will denote with $e_p(x)$. (Note that for every point in $X'$ we have $e_p(x) = 1$, and that there could be points in $p\inv(B)$ for which $e_p(x) = 1$.) The \emph{degree} of $p$ is the degree of $p'$, i.e. the number of points in the preimage of any point in $Y'$.
\end{defn}

Note that automatically $n_x > 0$ for each $x$. However, given $y \in B$, $n_x$ can vary when $x \in p\inv(y)$ and it could be 1 for some point in $p\inv(y)$ and strictly larger than 1 for some other. Also, tautologically, non-constant holomorphic maps between Riemann surfaces are branched covers.

Recall the following theorem, known as the Hurwitz formula.

\begin{thm}[Hurwitz formula]
If $p \co X \to Y$ is a branched cover of degree $d$, then
\[
\chi(X) = d\chi(Y) - \sum_{x\in X} (e_p(x)-1).
\]
\end{thm}

\begin{proof}[Sketch of proof]
Call $B \subset Y$ the branching set of $Y$. We have:
\[
\chi(X) = \chi(X') + \chi(p\inv(B)), \quad\quad \chi(Y) = \chi(Y') + \chi(B).
\]
Since $p'$ is an honest cover, we also have $\chi(X') = d\chi(Y')$. Now observe that for every point $y\in B$,
\[
\chi(p\inv(y)) = d - \sum_{p(x) = y} (e_p(x)-1) = d\chi(\{y\}) - \sum_{p(x) = y} (e_p(x)-1).\qedhere
\]
\end{proof}

\begin{proof}[First proof of the adjunction formula]
By Theorem~\ref{t:allcurvesareiso}, we can choose the curve we prefer. Let us pick the curve $F$ of degree $d$ defined by the equation $x^d+y^d-z^d = 0$. (Exercise: check that it is non-singular.) We want to construct a branched cover, so we project $\CP^2\setminus\{(1\sd 0\sd 0)\}$ onto the $(y\sd z)$-line by sending $(x \sd y \sd z)\mapsto(y \sd z)$. (Exercise: check that this is well-defined and holomorphic.) Now restrict this cover to the curve $F$, which gives us a holomorphic map $p\co F\to \CP^1 \cong S^2$.

We claim (exercise!) that each of the $d$ points $(e^{2k\pi i/d}\sd 1)$ has exactly one preimage, while every other point has exactly $d$ preimages. This means that $\deg p = d$ and that there are $d$ points in $F$ for which $e_p(x) = d$ and every other point is $1$. Using the Hurwitz formula:
\[
2 - 2g(F) = d\chi(\CP^1) - \sum_{x\in F} (e_p(x)-1) = 2d - d(d-1) = 3d-d^2,
\]
from which the statement follows.
\end{proof}

\begin{proof}[Sketch of a second proof of the adjunction formula]
By Theorem~\ref{t:allcurvesareiso}, we can choose the curve we prefer. We pick a very singular curve $C_0$, given as the product of $d$ generic lines, defined by a polynomial $f$ which splits in $d$ distinct linear factors, and then we slighly deform it by taking $f+\epsilon g$ for some (generic) polynomial $g$ and some small $\epsilon$. Call $C$ the zero-set of this deformed polynomial.

$C_0$ is the union of $d$ generic lines, which therefore meet at $d \choose 2$ double points. The effect of the deformation is to make the curve smooth by replacing the neighbourhood of a double point (which looks like $\{xy = 0\}$) by a deformed equation $\{xy = \eta\}$. The net effect is to replace two discs meeting at one point with an annulus (why?\footnote{This will be clarified in the next section.}).

At the level of Euler characteristic, each time we are replacing something of Euler characteristic $1$ (since two discs meeting at a point are contractible) by something of Euler characteristic $0$ (because it is an annulus). Moreover, since $C_0$ is a union of $d$ 2-spheres meeting at $d \choose 2$ double points, $\chi(C_0) = 2d- {d\choose 2}$ (why?). So:
\[
\chi(C) = \chi(C_0) - {d \choose 2} = 2d - 2{d \choose 2} = 3d-d^2,
\]
and again we are done.
\end{proof}

\begin{proof}[Sketch of a third proof of the adjunction formula]
Pick any non-singular curve $C$ and a generic point $q$ in $\CP^2$. The set of tangents that can be drawn from the point to the curve is given by the resultant of the defining polynomial of $C$, so (since $q$ is chosen generically) it consists of $d(d-1)$ lines that are simply-tangent to $C$. We now apply the Hurwitz formula to the projection from the point $q$, which is a degree-$d$ branched cover with $d(d-1)$ branching points, each of which has a single preimage with ramification index $2$ and all other preimages with ramification index $1$. We get:
\[
\chi(C) = d\chi(S^2) -  \sum_{x\in F} (e_p(x)-1) = 2d - d(d-1)\cdot 1 = 3d-d^2.\qedhere
\]
\end{proof}

\begin{rmk}
Any other proof of the adjunction formula can be used in reverse to prove the following fact: given a curve $C\subset \CP^2$, for a generic point $q$ in $\CP^2$ there exist exactly $d(d-1)$ lines passing through $q$ and tangent to $C$.
\end{rmk}

To close things out, we sketch the ``best'' (or at least the most algebro-geometric and most general) proof of the adjunction formula.

\begin{proof}[Sketch of a fourth proof of the adjunction formula]
The restriction $\rho$ of the tangent bundle of $\CP^2$ to $C$ splits as a direct sum of two complex line bundles: the tangent bundle $\tau$ of $C$ and the normal bundle $\nu$ to $C$. Since for line bundles the first Chern class is the Euler class, $\langle c_1(\tau), [C] \rangle = \chi(C)$, and $\langle c_1(\nu), [C] \rangle = C\cdot C = d^2$. On the other hand, $\langle c_1(\rho), [C] \rangle = d\langle c_1(\rho), h\rangle$. Let $k := \langle c_1(\rho), h\rangle$.

By Whitney's formula, $c_1(\rho) = c_1(\tau \oplus \nu) = c_1(\tau) + c_1(\nu)$, so:
\[
dk = \chi(C) + d^2.
\]
Now it is enough to compute $k$. There are two ways of doing this: the grown-up way is to prove that $c_1(\rho) = c_1(\det \rho)$, where $\det \rho = \rho \wedge \rho$ is a line bundle\footnote{If we did the same construction starting with the cotangent bundle instead of the tangent bundle, we would obtain the so-called \emph{canonical bundle} of $\CP^2$. This construction extends to arbitrary non-singular varieties.} on $C$, then sit down, find a section of the bundle $\det \rho$ and count how many zeros it has on $C$. The lazy way is to use the formula above for a curve of degree 1 (which we know is a 2-sphere, so it has $\chi = 2$), and get $k = 3$. Either way, we obtain that $\chi(C) = 3d-d^2$.
\end{proof}

A question for the less attentive reader: where did we use the fact that $C$ is a complex curve, and not just an arbitrary embedded surface, in this last proof?

\begin{rmk}
The fourth proof is a lot more general, in that it shows that there is a general relationship, for a curve $C$ in a complex surface $S$, between three quantities: $\langle c_1(TS), [C]\rangle$, $\chi(C)$, and $C\cdot C$:
\[
\langle c_1(TS), [C]\rangle = \chi(C) + C \cdot C.
\]
An algebraic geometer would use the canonical class $K_S$ instead of $c_1(TS)$ and would write:
\[
K_S\cdot C + \chi(C) + C\cdot C = 0,
\]
which they would recognise as \emph{the} adjunction formula.
\end{rmk}


\section{Singularities and the adjunction formula}\label{s:singularadjunction}

The goal of this section is to give a refinement of the adjunction formula for singular curves, as well as a criterion for obstructing the existence of curves with a given degree and certain configurations of singularities.

We want to study the local behaviour of singularities and their smoothings. Recall from Section~\ref{s:curves} that for every point $p$ on a curve $C \subset \C^2$ there exists a small Euclidean ball $B$ centred at $p$ such that:
\[
(B,C\cap B) \simeq (\Cone(\del B), \del B \cap C).
\]

Let $C = \{f = 0\}$ be a curve in $\C^2$ passing through the origin. Fix $\epsilon > 0$ so that the ball $B$ of radius $\epsilon$ and centred at the origin satisfies the cone condition above. Then for $\eta \in \C^*$ of sufficiently small absolute value, $\{f = \eta\} \cap B$ is a smoothly embedded and connected surface whose boundary is isotopic to $\del B \cap C$. It turns out that the isotopy class of this surface is independent of the chosen $\eta$.

\begin{defn}
In the notation above, the \emph{Milnor fibre} $M_{(C,p)}$ of the singularity of the singularity of $C$ at $p$ is the surface $\{f = \eta\}$. The \emph{Milnor number} $\mu_{(C,p)}$ of $(C,p)$ is $b_1(M_{(C,p)})$.
\end{defn}

\begin{thm}[Milnor]
The Milnor fibre $M_{(C,p)}$ can be pushed into $S^3$ (relative to its boundary) and gives a minimal-genus Seifert surface for the link $L_{(C,p)}$ of the singularity. This minimal-genus surface is the page of a fibration $S^3\setminus L_{(C,p)} \to S^1$ (and of a fibration $B^4\setminus C \to S^1$).
\end{thm}

Given a singular point $(C,p)$, we also denote by $\beta_{(C,p)}$ the number of branches of $C$ at $p$, that is the number of components of the link of $(C,p)$. Informally, the number of branches is the number of local irreducible components of $C$ that meet at the point $p$.

\begin{exe}
Show that the Milnor fibre of the double point $\{xy = 0\}$ is an annulus, so that its Milnor number is 1. Show that it has two branches.
\end{exe}

\begin{exe}\label{exe:torusknots}
Show that the Milnor fibre of the singularity of $\{x^a = y^b\}$ at the origin has $\beta = \gcd(a,b)$. Show that if $\gcd(a,b) = 1$, then $\mu = (a-1)(b-1)$. (First hint: the link of the singularity is a torus link. Second hint: the genus of torus can be computed by giving an explicit lower bound and an explicit upper bound, or it can be computed by exhibiting a fibration and using the fact that fibres are genus-minimising.)
\end{exe}

Since the Milnor fibre is independent of the chosen deformation, when we have a projective curve $C$ and we add a generic deformation, we are replacing the neighbourhood of each point by a non-singular surface with $b_1 = \mu$ and with $\beta$ boundary components. In particular, we replace a contractible piece of $C$ with a piece with Euler characteristic $1-\mu$.

Recall from Theorem~\ref{t:holomorphicimage} that every irreducible complex curve $C$ is the image of a holomorphic map from a connected Riemann surface $\Sigma$ into $\CP^2$. This map is essentially unique if we impose that the map is somewhere one-to-one. The genus $\Sigma$ is then called the \emph{geometric genus}. We denote by $\chi_g(C)$ the Euler characteristic of $\Sigma$.

\begin{exe}
Show that the quantity $\chi_g(C)$ agrees with the Euler characteristic of $C$ as a subspace of $\CP^2$ if and only if $C$ has only cuspidal singularities. Show that in general the difference between the two is $\sum_p (\beta_{(C,p)} - 1)$.
\end{exe}

To sum up, using the adjunction formula, the independence of the deformation choices, and the exercise above, we easily prove the following generalisation of the adjunction formula.

\begin{thm}[Singular adjunction formula]\label{t:singularadjunction}
Let $C$ a curve with irreducible components $C_1, \dots, C_n$. Then
\begin{equation}\label{e:singadj}
\sum \chi_g(C_k) = 3d-d^2 + \sum_p (\mu_{(C,p)} + \beta_{(C,p)} - 1).
\end{equation}
\end{thm}

A particularly nice case is that of \emph{rational cuspidal curves}. These are irreducible curves whose singularities are all cusps, i.e. they have connected link. Equivalently, they are curves such that $C$ \emph{as a subset of} $\CP^2$ is homeomorphic to a 2-sphere. Yet another way of saying this is that $C$ is (isotopic to) a PL-embedded 2-sphere. In this case, the adjunction formula simplifies as follows.

\begin{cor}
Let $C$ be a rational cuspidal curve in $\CP^2$. Then
\[
\sum g(L_{(C,p)}) = \frac12(d-1)(d-2).
\]
In particular, if $C$ has a single singularity which is of type $T(a,b)$, then
\[
(a-1)(b-1) = (d-1)(d-2).
\]
\end{cor}

Rational cuspidal curves with a single singularity which is of type $T(a,b)$ are actually classified.

\begin{thm}[Fern\'andez de Bobadilla, Luengo, Melle Hern\'andez, N\'emethi~\cite{FdBLMHN}]
If the rational cuspidal curve $C \subset \CP^2$ has degree $d$ and a singularity of type $T(a,b)$, then $(a,b,d)$ is one of the following:
\[
\begin{array}{lcclccl}
(a,a+1,a+1), & & & (F_k, F_{k+4}, F_{k+2}), & & &  (3,22),\\
(a,4a-1,2a), & & & (F_k^2, F_{k+2}^2, F_{k+1}), & & & (6,43).\\
\end{array}
\]
where $a \ge 2$, $k \ge 1$ is odd, and $F_0, F_1, F_2, \dots$ are the Fibonacci numbers $0,1,1,\dots$
\end{thm}

\begin{rmk}
The question of classification of these objects has been studied also using Heegaard Floer homology~\cite{BL, Liu-classification, BCG1}, in the symplectic setup~\cite{GS1, GKutle}, and in the PL setup~\cite{AGLL}.
\end{rmk}

The number of branches is clearly an invariant of the singularity that only depends on its link. In fact, this is true (albeit not obvious) also for the Milnor number. Moreover, for a given $N$, there are finitely many topological types of singularities with $\mu_{(C,p)} + \beta_{(C,p)} \le N$.

A corollary of this fact and of the singular adjunction formula is the following.

\begin{cor}
Fix a degree $d$. There are only finitely many possible collections of singularities that appear as singularities of a curve of degree $d$.
\end{cor}

\begin{proof}
Since the Milnor fibre is a compact, connected surface with boundary, we have $\mu_{(C,p)} + \beta_{(C,p)} - 1 \ge 0$. Moreover, this quantity vanishes if and only if the Milnor fibre of $(C,p)$ is a disc, which happens only when $p$ is a non-singular point of $C$, and is otherwise an even integer. That is to say, the right-hand side of~\eqref{e:singadj} is bounded from below by $3d-d^2$.

On the other hand, the left-hand side of~\eqref{e:singadj} is bounded from above by $2d$, which can be attained when $C$ has exactly $d$ components, each of which is a 2-sphere (that is, $\chi_g(C_k) = 2$ for each $k$).

So $\sum_{p\in\Sing(C)} (\mu_{(C,p)} + \beta_{(C,p)} - 1)$ is bounded from above by $d^2-d$. Since the quantity on the left-hand side is a sum of positive terms, $C$ can have at most $\frac12(d^2-d)$ singular points. (Note that this maximum is attained by a line arrangement with only double points.)

Since there are finitely many singularities with $\mu_{(C,p)} + \beta_{(C,p)} - 1 \le N$ for any given $N$, the statement now follows.
\end{proof}

The number of possible combinatorial types satisfying adjunction can be estimated, but it grows very fast. For instance, let us restrict to curves with singularities of type $T(2,2k+1)$. Then the number of combinatorial types of curves of degree $d$ with only singularities of type $T(2,2k+1)$ is $p_0 + p_1 + \dots + p_{\frac12(d-1)(d-2)}$ (exercise!). Here, $p_n$ is the number of unordered partitions of $n$, i.e. the number of ways of writing $n$ as a sum of positive integers. Hardy and Ramanujan~\cite{HardyRamanujan}, and independently Uspensky~\cite{Uspensky}, proved that $p_n \sim \frac1{4n\sqrt3}\exp\Big(\pi\sqrt{\frac{2n}{3}}\Big)$, so the growth is at least exponential in the degree.


\section{Blow-ups and singular curves}\label{s:blow-ups}

Let us consider the following complex surface:
\[
X = \{(x,y; s \sd t) \mid xt = ys\} \subset \C^2 \times \CP^1,
\]
where $\C^2$ has coordinates $(x,y)$ and $\CP^1$ has coordinates $(s \sd t)$.

One easily verifies that $X$ is a non-singular complex surface (exercise!). The projection $\C^2 \times \CP^1 \to \C^2$ restricts to a map $\pi\co X \to \C^2$. This map is rather curious: it is an isomorphism away from the preimage of the origin, and it contains all of $\{(0,0)\}\times \CP^1$ (which is projected down to a single point).

\begin{defn}
The pair $(X,\pi)$ is called the \emph{blow-up} of $\C^2$ at the origin. The preimage $\pi\inv(0,0)$ of the origin is called the \emph{exceptional divisor} of the blow-up.
\end{defn}

Often we will omit the map $\pi$ and just say that $X$ is the blow-up of $\C^2$ at the origin. Note that the restriction of the other projection endows $X$ with the structure of a $\C$-bundle over $\CP^1$, and in particular $X$ deformation-retracts onto $\CP^1$.

By choosing local coordinates around any point $x$ of a non-singular complex surface $S$, we can blow-up $S$ at $p$, and obtain a new surface $\tilde S$ that has an extra $\CP^1$ instead of $p$, but otherwise looks a lot like $S$.

\begin{prop}
The exceptional divisor of the blow-up $\tilde S$ of $S$ at a point is a smoothly embedded sphere of self-intersection $-1$. In particular, $\tilde S$ is diffeomorphic to $S \# \CPbar$.
\end{prop}

Before going into the proof, let us look more closely at $X$. Take a curve $C \subset \C^2$. If $C$ does not contain the origin, then we can lift $C$ to $X$ by simply taking the preimage $\pi\inv(C) \subset X$, and the restriction of $\pi$ induces an isomorphism $\pi\inv(C) \to C$. If $C$ contains the origin, on the other hand, $\pi\inv(C)$ always contains the exceptional divisor $E$. Let us do something a bit different: look at the closure of $\pi\inv(C) \setminus E$ inside $X$. The latter is called the \emph{strict transform} (or sometimes \emph{proper transform}) of $C$ in $X$, while $\pi\inv(C)$ is called the \emph{total transform} of $C$ in $X$.

\begin{exe}\label{exe:lifting-line}
Show that if $\ell \subset \CP^2$ is a line and $p\in \ell$ is a point, the strict transform of $\ell$ in the blow-up of $\CP^2$ meets the exceptional divisor transversely at one point\footnote{Hint: explicitly lift the curve in a chart centred at a convenient point of the exceptional divisor.}.
\end{exe}

\begin{exe}
Let $C \subset \CP^2$ be the curve $C = \{x^pz^{q-p} - y^q = 0\}$, where $0 < p < q$ and $\gcd(p,q) = 1$. Show that the strict transform $C$ in the blow-up of $\CP^2$ at $(0\sd 0 \sd 1)$ has a singularity of type $\{x^p - y^{q-p} = 0\}$.
\end{exe}

\begin{proof}[Sketch of proof]
The blow-up of $S$ at a point replaces an open ball neighbourhood of a point with the surface $X$ from above, so it is the connected sum of $S$ with \emph{something}, and the homology groups of this \emph{something} are the same as those of $\CP^2$ (why?). In particular, $H_2(\tilde S) = H_2(S) \oplus \Z$, where the splitting is orthogonal, and it is given by:
\[
\xymatrix{
H_2(S) \oplus \Z \ar[r]^-{\sim} & H_2(S\setminus \{p\}) \oplus \Z\ar[r] & H_2(\tilde S)\\
(C, b)\ar@{|-{>}}[r] & (\iota\inv(C), b) \ar@{|-{>}}[r] & \tilde\iota\iota\inv(C) + b[E].
}
\]

Here $\iota\co H_2(S\setminus\{p\}) \to H_2(S)$ and $\tilde\iota\co H_2(S\setminus\{p\}) \to H_2(\tilde S)$ are the maps induced by the inclusion of $S\setminus\{p\}$ into $S$ and $\tilde S$, respectively.

Let us blow up $\CP^2$ at the point $(0\sd 0 \sd 1)$, to obtain the surface $T$. Call $E$ the exceptional divisor of the blow-up. Let us look at the proper transforms of the two lines $\ell_1 = \{x = 0\}$ and $\ell_2 = \{y = 0\}$, in $\C^2$. The proper transforms of $\ell_1$ and $\ell_2$ are $\tilde\ell_1  = \{(0,t;0\sd 1) \mid t \in \C\}$ and $\tilde \ell_2 = \{(t,0;1\sd 0) \mid t \in \C\}$ (exercise!). In particular, $\tilde\ell_1$ and $\tilde\ell_2$ are disjoint and they each intersect $E$ transversely once (exercise!).

Let us write $[\tilde\ell_1] = ah - be$, where $h, e\in H_2(T) \cong \Z^2$ form the orthogonal basis described above (and $h$ is the homology class of a complex line in $\CP^2$). Since $\tilde\ell_2$ is isotopic to $\tilde\ell_1$ (by lifting any isotopy between $\ell_1$ and $\ell_2$ through lines containing $(0\sd 0\sd 1)$), $[\tilde\ell_2] = [\tilde\ell_1] = ah - be$. Moreover, $\tilde\ell_1$ and $\tilde\ell_2$ are disjoint, so
\[
0 = [\tilde\ell_2]\cdot[\tilde\ell_1] = (ah-be)^2 = a^2 + b^2(e\cdot e)
\]
Since both $\ell_1$ and $\ell_2$ meet every line not passing through the origin transversely and positively once (because they do in $\CP^2$), then
\[
1 = [\ell_1]\cdot h = (ah+be)\cdot h = a,
\]
so $a = 1$. It follows that
\[
0 = a^2 + b^2(e\cdot e) = 1 + b^2(e\cdot e),
\]
so $e\cdot e = -1$ and\footnote{In fact, since $\tilde\ell_1$ and $E$ intersect transversely and positively once, $b=-1$.} $b^2 = 1$.
\end{proof}

\begin{rmk}
The one presented above is \emph{not} the ``good'' proof of the statement, but it's a hands-on and rather elementary proof. A better proof would involve computing the first Chern class of the bundle $X\to \CP^1$ described above: this bundle is dual to a bundle which has a section with one zero (of multiplicity $1$), which proves the first half of the statement.
We chose to stick to this line of proof since it only uses considerations in 4-dimensional topology and avoids the use of characteristic classes.
\end{rmk}

\begin{rmk}
Note that, since $e\cdot e = -1$, there can be no holomorphic perturbation of $E$ within $\tilde S$: the exceptional divisor is a \emph{rigid} complex curve.
\end{rmk}

\begin{defn}
Let $S$ be a complex surface and $C, D \subset S$ be two curves.
\begin{itemize}
\item Pick a point $p \in C$, blow up $S$ at $p$, and call $E$ the exceptional divisor of the blow-up.
The \emph{multiplicity} of $C$ at $p$ is the multiplicity of intersection of $E$ with the strict transform of $C$.

\item Pick a point $p \in C \cap D$. The \emph{local multiplicity of intersection} $(C\cdot D)_p$ of $C$ and $D$ at $p$ is defined as follows.
Choose a small round neighbourhood $B$ of $p$ such that $(B, C\cup D \cap B)$ is the cone over the link of $C\cup D$ at $p$ and a deformation $\{C_t\}_{t\in [0,1]}$ of $C$ such that $C_0 = C$, $C_t$ is transverse to $\del B$ for every $t$, and $C_1$ and $D$ intersect transversely (at non-singular points of $C_1$ and of $D$).
Then we define $(C\cdot D)_p$ as the number of intersections between $C_1$ and $D$ inside $B$.
\end{itemize}
\end{defn}

\begin{exe}
The local multiplicity of intersection $(C\cdot D)_p$ has a 3-dimensional interpretation in the following terms.
Consider the link of $C\cup D$ at $p$: this is an oriented link $L$ in $S^3$, which is partitioned into two sub-links $L = L_C \cup L_D$, corresponding to the components of $C$ and of $D$.
Then $(C\cdot D)_p = \lk(L_C, L_D)$, where the linking number here is the total linking number (that is, we sum all linking numbers between pairs of components, one from $L_C$ and one from $L_D$).
\end{exe}

\begin{prop}
The multiplicity of $C$ at a point $p$ is the local multiplicity of intersection of a generic line through $p$ with $C$. In particular, the multiplicity of $C$ at a non-singular point is $1$.
\end{prop}

\begin{exe}
Show that this definition of multiplicity agrees with the one given in Section~\ref{s:curves} above for singular points of line arrangements.
\end{exe}

\begin{proof}
Consider the proper transform $\tilde C$ of $C$ and the proper transform $\tilde\ell$ of a generic line $\ell$ through $p$. From Exercise~\ref{exe:lifting-line}, we know that $\tilde \ell$ meets the exceptional divisor $E$ transversely at a single point. Before proving the claim, let us see how it helps us conclude the proof.

Since $\tilde \ell$ meets $E$ transversely at a point, $[\tilde \ell] = [\ell] - E$ (why?). On the other hand, by definition of multiplicity, $[\tilde C] = [C] - mE$, where $m$ is the multiplicity of $C$ at $p$. Now, since $\tilde C$ intersects $E$ at finitely many points and $\ell$ was chosen to be generic, $\tilde\ell$ and $\tilde C$ do not meet along $E$. Therefore, the $\tilde\ell$ and $\tilde C$ are disjoint in a neighbourhood $N$ of $E$, so their intersection number within this neighbourhood, which we denote by $(\tilde\ell \cdot \tilde C)_N$, vanishes. Thus we have:
\[
0 = (\tilde\ell \cdot \tilde C)_N = ((\ell - E)\cdot(C - mE))_N = (\ell\cdot C)_p + m(E^2)_N = (\ell\cdot C)_p - m,
\]
from which the statement follows.
\end{proof}

\begin{exe}
Let $C \subset \CP^2$ be the curve $C = \{x^pz^{q-p} - y^q = 0\}$, where $0 < p < q$ and $\gcd(p,q) = 1$. Compute the multiplicity of $C$ at $(0\sd 0 \sd 1)$ and at $(1\sd 0 \sd 0)$.
\end{exe}

\section{Branched covers: dimension 4}\label{s:Fano}

In this section we will consider the Fano arrangement, represented schematically in Figure~\ref{f:Fano}.

\begin{figure}
\centering
\mfig{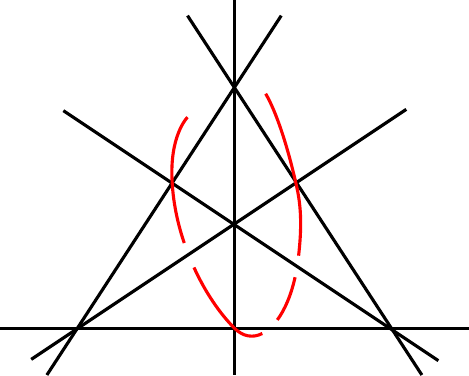}
\caption{The seven lines of the Fano configuration. The red line is one single line, even though it is drawn as a curve.}\label{f:Fano}
\end{figure}

\begin{defn}
The \emph{Fano arrangement} is the arrangement of seven lines meeting at seven triple points.
\end{defn}

The Fano arrangement is realised over the field $\F_2$ with two elements: the projective plane $\F_2\P^2$ over $\F_2$ contains seven points and seven lines, and each point is contained in exactly three lines (since $\F_2\P^1$ has cardinality 3).

The main result of this section is the following non-existence result.

\begin{thm}\label{t:Fano}
There exist no complex line arrangement with the combinatorial type of the Fano arrangement.
\end{thm}

We will present three proofs (of many that are available). The first is folklore, and works in the complex category (or, in fact, over any field of characteristic different than 2), while the other two work for surfaces with conical singularities.

\begin{proof}[A sketch of the rigid proof of Theorem~\ref{t:Fano}]
Let $\F$ be a field and suppose that the Fano arrangement is realised in $\F\P^2$. Up to the action of $\Aut(\F\P^2) = PGL(3,\F)$, we can suppose that the points labelled $1$, $2$, $4$, and $6$ in Figure~\ref{f:Fano-with-labels} are sent to $(0\sd 0 \sd1)$, $(0 \sd 1 \sd 0)$, $(1 \sd 0 \sd 0)$, and $(1 \sd 1 \sd 1)$, respectively. The six lines passing through all possible pairs among these four points are determined by the equations $x=0$, $y=0$, $z=0$, $x=y$, $y=z$, and $z=x$. Therefore, the remaining three intersection points have coordinates $(1\sd 1\sd 0)$, $(1\sd 0\sd 1)$, and $(0\sd 1\sd 1)$. These points simultaneously satisfy the linear equation $ax+by+cz = 0$ if and only if $a+b=b+c=c+a=0$, which has a non-zero solution (in $a$, $b$, $c$) if and only if the characteristic of $\F$ is 2.
\end{proof}

The proof above shows something stronger than the statement: the only fields over which the Fano line arrangement can be realised are those of characteristic 2.

We want to give two topological proofs of Theorem~\ref{t:Fano}. For the first, we will need to generalise the notion of branched covers of surfaces (Definition~\ref{d:bc2}) to higher dimension (and especially to dimension 4). We will restrict to smooth branched covers whose branch locus is smoothly embedded, but the definition can be extended way beyond this setup~\cite{Fox-coverings}.

\begin{defn}
Let $X$ and $Y$ be two compact, oriented $n$-manifolds, and $p\co X\to Y$ a smooth map. $p$ is called a \emph{branched cover} if there exists a smooth codimension-2 submanifold $B \subset Y$ such that, calling $X' = X \setminus p\inv(B)$, $Y' = Y \setminus B$, $p'\co X' \to Y'$ the restriction of $p$ to $X'$ (with the restriction on its codomain, too):
\begin{itemize}\itemsep 0pt
\item $p'$ is a covering map (i.e. a local homeomorphism);
\item for every point $x \in p\inv(B)$, there are coordinates $B^{n-2} \times \Delta$ around $x$ and $B^{n-2}\times \Delta$ around $p(x)$ in which $p\inv(B)$ is $U\times\{0\}$, $B$ is $V\times\{0\}$, and $p$ is $\id \times \mu_{n_x}$.
\end{itemize}
$B$ is called the \emph{branching set}, and $p'$ is the \emph{cover} associated to $p$. We also say, by a slight abuse of terminology, that $X$ is a branched cover of $Y$, or that $X$ is a cover of $Y$ branched over $B$. $n_x$ is called the \emph{index of ramification} of $p$ at $x$. (Note that for every point in $X'$ we have $n_x = 1$, and that there could be points in $p\inv(B)$ for which $n_x = 1$.) The \emph{degree} of $p$ is the degree of $p'$, i.e. the number of points in the preimage of any point in $Y'$. We will assume that for every $y \in B$ there exists $x\in p\inv(B)$ with $n_x > 1$.
\end{defn}

We will look at the special case of \emph{double} covers, i.e. those whose degree is 2. In this case, the restriction of $p$ to $p\inv(B)$ is a diffeomorphism onto $B$. Similarly to the Riemann--Hurwitz formula, we can compute the Euler characteristic of $X$ in terms of those of $Y$ and $B$.

\begin{prop}\label{p:chi}
If $X$ and $Y$ are closed $n$-manifolds and $p\co X\to Y$ is a double cover branched over $B$, then
\[
\chi(X) = 2\chi(Y) - \chi(B).
\]
\end{prop}

\begin{proof}
If $n$ is odd, then $B$ is also odd-dimensional and we have $\chi(X) = \chi(Y) = \chi(B) = 0$, so the formula holds.

If $n$ is even, the boundary of a tubular neighbourhood of $B$ in $Y$ and of $p\inv(B)$ in $X$ is odd-dimensional, so its Euler characteristic vanishes\footnote{It also vanishes because it is a circle bundle.}. By additivity of the Euler characteristic, $\chi(X) = \chi(X') + \chi(p\inv(B)) = \chi(X') + \chi(B)$ (since $p$ is a diffeomorphism between $p\inv(B)$ to $B$) and $\chi(Y) = \chi(Y') + \chi(B)$.

By multiplicativity of the Euler characteristic, $\chi(X') = 2\chi(Y')$. Putting everything together, we have:
\[
\chi(X) = \chi(X') + \chi(B) = 2\chi(Y') + \chi(B) = 2(\chi(Y)-\chi(B)) + \chi(B) =  2\chi(Y) - \chi(B).\qedhere
\]
\end{proof}

Moreover, if $p$ is a double branched cover, then the associated cover $p'$ is a double cover, so $\pi_1(X')$ is an index-2 subgroup of $\pi_1(Y')$, and in particular it is normal.

\begin{prop}\label{p:exist-double-cover-H1}
Fix a smooth manifold $Y$ and a smooth codimension-$2$ submanifold $B\subset Y$. There exists a double cover $p \co X \to Y$ branched over $B$ if and only if there exists a homomorphism $H_1(Y\setminus B) \to \Z/2\Z$ mapping the meridian of each component of $B$ to $1$.
\end{prop}

\begin{proof}
We first prove that if there is a double cover $p\co X \to Y$ branched over $B$ then there is a homomorphism $H_1(Y\setminus B) \to \Z/2\Z$ as above.
A degree-$2$ cover corresponds to a cover $p' \co X' \to Y\setminus B$ of degree 2, so to an index-2 subgroup $p'_*(\pi_1(X')) < \pi_1(Y \setminus B)$.
Such a subgroup is automatically normal and has $\pi_1(Y \setminus B)/p'_*(\pi_1(X')) \cong \Z/2\Z$.
That is to say, $p'_*(\pi_1(X'))$ is the kernel of a surjection $\phi \co \pi_1(Y \setminus B) \to \Z/2\Z$.
Since $\Z/2\Z$ is Abelian, $\phi$ factors through the Abelianisation of $\pi_1(Y\setminus B)$, i.e. $H_1(Y \setminus B)$.
In other words, there is a homomorphism $\overline{\phi} \co H_1(Y \setminus B) \to \Z/2\Z$ such that the following diagram commutes:
\[
\xymatrix{
\pi_1(Y \setminus B) \ar[r]_{\phi} \ar[d]		& \Z/2\Z	\\
H_1(Y\setminus B) \ar[ur]_{\overline{\phi}_p}	&\\
}
\]
where the vertical map is the natural homomorphism $\pi_1 \to H_1$ (which is essentially the Abelianisation map).
Since the lift of a meridian of $B$ is \emph{half} a meridian of $p\inv(B)$, the associated map sends the meridian to $1 \in \Z/2\Z$.

For the other direction, suppose that we have a homomorphism $\phi\co H_1(Y\setminus B) \to \Z/2\Z$ that sends each meridian of $B$ to $1$.
Then there is an index-2 subgroup of $\pi_1(Y\setminus B)$, namely the kernel of the composition $\pi_1(Y\setminus B) \to H_1(Y\setminus B) \to \Z/2\Z$, which in turn corresponds to a double cover $p'\co X' \to Y\setminus B$.
Consider a small tubular neighbourhood $N_i$ of a component $B_i$ of $B$.
$N_i$ is a disc bundle over $B_i$, and $Y_i := (p')\inv(\del N_i)$ is a circle bundle over $B_i$.
The restriction of $p'$ to $Y_i$ is a bundle map $Y' \to\del N_i$, which by the assumption on $\phi$ double covers each circle fibre of $\del N_i$.
We can therefore extend $p'|_{Y_i}$ to a double cover $\tilde N_i \to N_i$ from the disc bundle $\tilde N_i$ associated to $Y_i$, which branches on the 0-section of $\tilde N_i$ over $B_i$.
If we do this for each component of $B$, we compactify $p'$ to a double cover $p \co X \to Y$ which is branched precisely over $B$.
\end{proof}

We are now ready for the first topological proof, which is adapted from work of Ruberman and Starkston~\cite{RubermanStarkston}\footnote{In their proof, Ruberman and Starkston use the G-signature theorem to conclude, while we bypass it at the very end. Their argument generalises to other configurations, notably to finite projective planes.}.

\begin{proof}[Second proof of Theorem~\ref{t:Fano}]
Suppose that the Fano configuration is realised.
Blow up $\CP^2$ at the seven intersection points of the configuration, and consider the proper transforms of the seven lines. Each has been blown up three times, we have a collection of pairwise disjoint embedded spheres of self-intersection $-2$ in $Y = \CP^2 \#7\CPbar$. Fix generators $h, e_1, \dots, e_7$ for $H_2(Y)$, corresponding to the $+1$-sphere that generates $H_2(\CP^2)$ and to the seven $(-1)$-spheres that generate $H_2(7\CPbar)$.

More precisely, they are in the homology classes
\[
\begin{array}{llll}
h-e_1-e_2-e_3, & h-e_1-e_4-e_5, & h-e_2-e_4-e_7, & h-e_3-e_5-e_7,\\
h-e_1-e_6-e_7, & h-e_2-e_5-e_6, & h-e_3-e_4-e_6.
\end{array}
\]
Call $F_1, \dots, F_4$ the spheres in the homology classes of the top row and $F_5$, $F_6$, and $F_7$ the ones from the bottom row. This is represented in Figure~\ref{f:Fano-with-labels}. Note that $[F_1] + \dots + [F_4] \in H_2(Y)$ is a class that is divisible by $2$. So we can take the double cover of $Y$ branched over $B = F_1 \cup \dots \cup F_4$, $p \co X \to Y$.

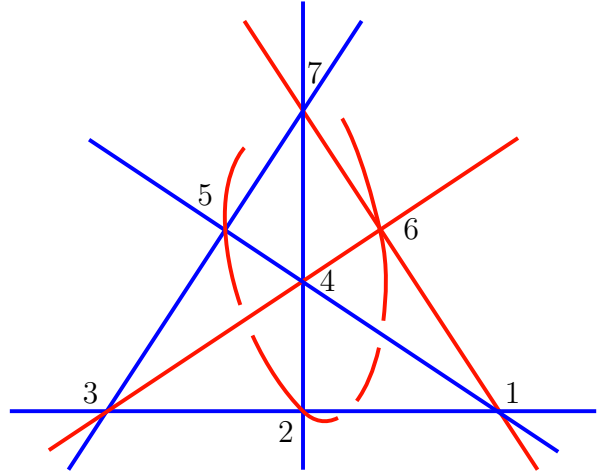
\begin{figure}
\centering
\labellist
\pinlabel $1$ at 193 30
\pinlabel $2$ at 106 15
\pinlabel $3$ at 31 30
\pinlabel $4$ at 122 73
\pinlabel $5$ at 75 106
\pinlabel $6$ at 154 93
\pinlabel $7$ at 117 153
\endlabellist
\mfig{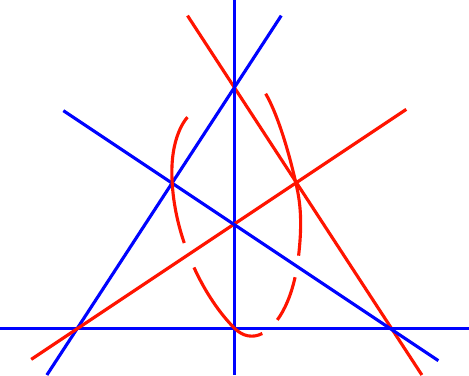}
\caption{We labelled the points in the configurations by $1,\dots,7$; the blue lines represent $F_1, \dots, F_4$, and the red lines $F_5$, $F_6$, and $F_7$}\label{f:Fano-with-labels}
\end{figure}

In $X$ we have the lifts $\tilde F_1, \dots, \tilde F_4$ of $F_1, \dots, F_4$, each of which is a sphere of self-intersection $-1$ (why? Exercise!). Each of $F_5$, $F_6$, $F_7$ is a \emph{sphere} that is disjoint from the branching set, so it lifts to two disjoint spheres $\tilde F_5^\ell$, $\tilde F_6^\ell$, and $\tilde F_7^\ell$ ($\ell = 1,2$), each of self-intersection equal to $-2$. This is because a map to a sphere with a basepoint has one lift for each preimage of the basepoint; the fact that the two preimages are disjoint follows from the fact that each of $F_5$, $F_6$, $F_7$ is embedded in $Y$ and $p$ is  a cover away from $B$ (so a local homeomorphism). The situation is schematically represented in Figure~\ref{f:Fano-double-cover}.

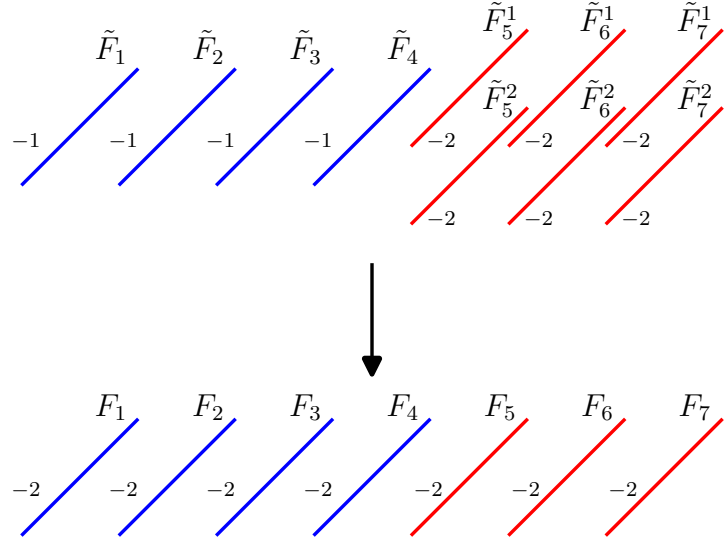
\begin{figure}
\labellist
\pinlabel $\tilde{F}_1$ at 70 378
\pinlabel $F_1$ at 70 100
\pinlabel $\vphantom{}_{-1}$ at 5 305
\pinlabel $\vphantom{}_{-2}$ at 5 36

\pinlabel $\tilde{F}_2$ at 145 378
\pinlabel $F_2$ at 145 100
\pinlabel $\vphantom{}_{-1}$ at 80 305
\pinlabel $\vphantom{}_{-2}$ at 80 36

\pinlabel $\tilde{F}_3$ at 220 378
\pinlabel $F_3$ at 220 100
\pinlabel $\vphantom{}_{-1}$ at 155 305
\pinlabel $\vphantom{}_{-2}$ at 155 36

\pinlabel $\tilde{F}_4$ at 295 378
\pinlabel $F_4$ at 295 100
\pinlabel $\vphantom{}_{-1}$ at 230 305
\pinlabel $\vphantom{}_{-2}$ at 230 36

\pinlabel $\tilde{F}^1_5$ at 370 398
\pinlabel $\tilde{F}^2_5$ at 370 338
\pinlabel $F_5$ at 370 100
\pinlabel $\vphantom{}_{-2}$ at 325 305
\pinlabel $\vphantom{}_{-2}$ at 325 245
\pinlabel $\vphantom{}_{-2}$ at 315 36

\pinlabel $\tilde{F}^1_6$ at 445 398
\pinlabel $\tilde{F}^2_6$ at 445 338
\pinlabel $F_6$ at 445 100
\pinlabel $\vphantom{}_{-2}$ at 400 305
\pinlabel $\vphantom{}_{-2}$ at 400 245
\pinlabel $\vphantom{}_{-2}$ at 390 36

\pinlabel $\tilde{F}^1_7$ at 520 398
\pinlabel $\tilde{F}^2_7$ at 520 338
\pinlabel $F_7$ at 520 100
\pinlabel $\vphantom{}_{-2}$ at 475 305
\pinlabel $\vphantom{}_{-2}$ at 475 245
\pinlabel $\vphantom{}_{-2}$ at 465 36
\endlabellist
\centering
\mlfig{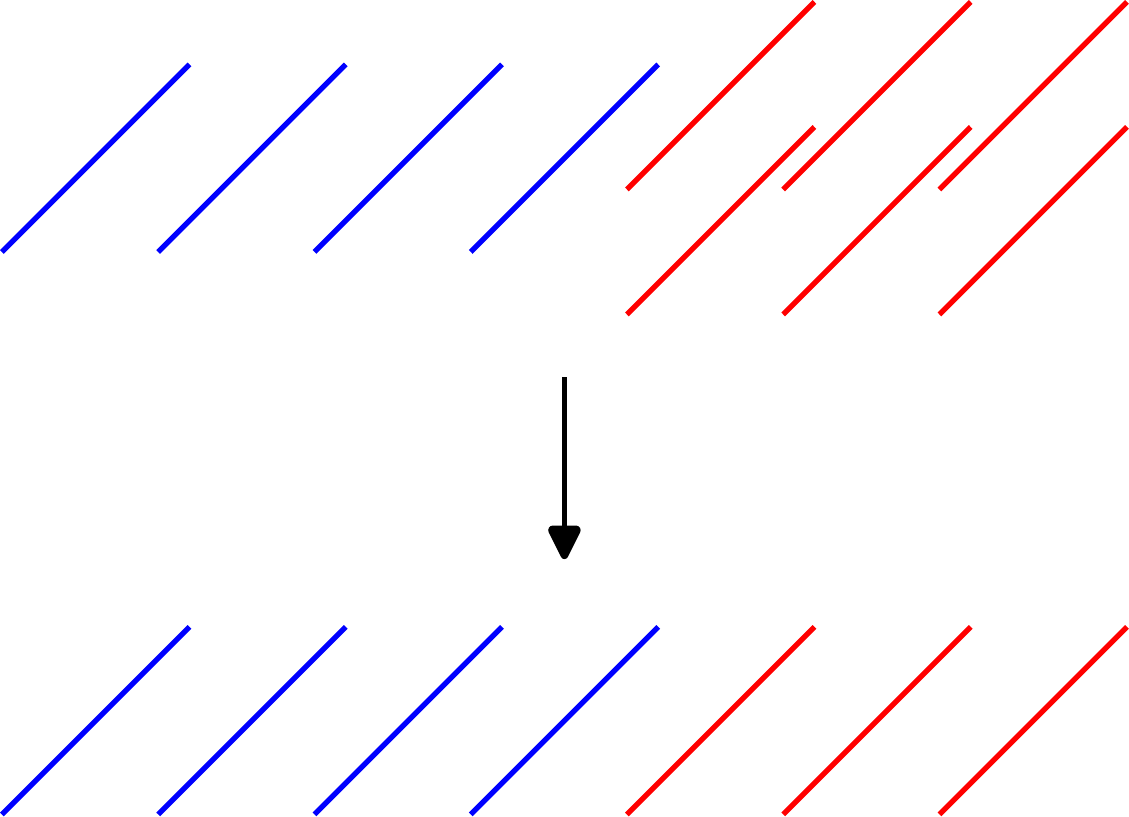}
\caption{The double cover of $\CP^2 \# 7\CPbar$ branched over $F_1 \cup \dots \cup F_4$, and the configuration of spheres coming from $F_1, \dots, F_7$ therein.}\label{f:Fano-double-cover}
\end{figure}

Summarising, we have ten classes, the $[\tilde F_j]$ (for $j = 1,\dots, 4$) and the pairs $[\tilde F_k^\ell]$ (for $k = 5,6,7$, $\ell = 1,2$), that are pairwise orthogonal (since they are disjoint) and negative. That is,
\begin{equation}\label{e:too-much-negativity}
b_2^-(X) \ge 10.
\end{equation}

On the other hand, $H_1(Y \setminus B) = \Z/2\Z$: this follows in the usual way. To wit, let $N$ be an open tubular neighbourhood of $B$.
From homotopy invariance, Poincar\'e--Lefschetz duality, and excision we deduce:
\[
H_1(Y \setminus B) \cong H_1(Y \setminus N) \cong H^3(Y \setminus N, \del N) \cong H^3(Y,B),
\]
From the long exact sequence of the pair $(Y,B)$ we have
\[
H^2(Y) \to H^2(B) \to H^3(Y,B) \to H^3(Y) = 0,
\]
and since $H_*(B)$ and $H_*(Y)$ are torsion-free, the map $H^2(Y) \to H^2(B)$ is the transpose of the map $H_2(B) \to H_2(Y)$ induced by the inclusion.
We have a basis for $H_2(B)$ given by the fundamental classes of the four components, and the basis $\{h,e_1,\dots,e_7\}$ for $H_2(Y)$ given above.
In these bases, $H_2(B) \to H_2(Y)$ is represented by the matrix:
\[
\left(
\begin{array}{rrrrrrrr}
1 & -1 & -1 & -1 & 0 & 0 & 0 & 0\\
1 & -1 & 0 & 0 & -1 & -1 & 0 & 0\\
1 & 0 & -1 & 0 & -1 & 0 & 0 & -1\\
1 & 0 & 0 & -1 & 0 & -1 & 0 & -1\\
\end{array}
\right),
\]
so its transpose represents the map $H^2(Y) \to H^2(B)$.
We can now check that this latter map has cokernel isomorphic to $\Z/2\Z$ (in fact, the coset containing of the fundamental class of any of the components of $B$ generates it).

Now recall that the Goldschmidt lemma~{\cite[Lemma~4.1 and following corollary]{HsiangSzczarba}} asserts that if $G$ is a group with finite cyclic Abelianisation and if $H$ is a subgroup of $G$ of prime power index, such that the quotient $G/H$ is still Abelian, then $H$ has finite Abelianisation. As a corollary, if a manifold $M$ has finite cyclic first homology group and $\tilde M$ is a cyclic cover of prime power order, then also $H_1(\tilde M)$ is finite\footnote{This is similar to the better-known fact that the double cover of $S^3$ branched over a knot is an $\F_2$-homology sphere, and that more generally the $p^r$-fold cyclic cover of $S^3$ branched over a knot is an $\F_p$-homology sphere (and in particular a rational homology sphere) for every prime power $p^r$.}. We apply this to the complement $Y'$ and its cover $X'$, and we obtain that $H_1(X')$ is finite. Since $H_1(X)$ is a quotient of $H_1(X')$ (why?), then it is also finite, and therefore $b_1(X) = 0$ and $\chi(X) = 2+b_2(X)$. By Proposition~\ref{p:chi}, $\chi(X) = 2\chi(Y) - \chi(B) = 20-8 = 12$, from which $b_2(X) = 10$. Since $b_2^-(X) \ge 10$, we should have $b_2^+(X) = 0$.

However, the preimage of a generic line (that is transverse to the branching set) is a connected surface of self-intersection $+2$ (why?), and in particular $b_2^+(X) > 0$, which gives a contradiction. 
\end{proof}

The second topological proof is a special case of a more general result, proved in~\cite{triplepoints}.

\begin{proof}[Sketch of a third proof of Theorem~\ref{t:Fano}]
As in the previous proof, suppose that such a configuration exists, and blow up $\CP^2$ at its seven singular points. The proper transform of the configuration comprises seven smoothly embedded spheres in $Y$, each of self-intersection $-2$. Tube the spheres together, to obtain an embedded sphere, which lives in the homology class $7h-3(e_1+ \dots + e_7)$.

Since this class $T$ has all odd coefficients in the base $h, e_1, \dots, e_7$, it is a \emph{characteristic} class (i.e. $T\cdot A \equiv A\cdot A \pmod 2$ for every $A \in H_2(Y)$). However, for every characteristic class $T$ in $Y$ represented by a smoothly embedded 2-sphere, $T\cdot T \equiv \sigma(Y) \pmod{16}$~\cite{KervaireMilnor}, but in our case
\[
-14 = T\cdot T \not\equiv \sigma(Y) = -6 \pmod{16},
\]
which gives a contradiction.
\end{proof}

The advantage of the last two proofs is that they work in a much more flexible category than the complex one: the proofs above work in the smooth category (i.e. when each ``line'' is a smoothly embedded 2-sphere and every two ``lines'' intersect transversely once), and they work for more general 4-manifolds than $\CP^2$: in fact, any 4-manifold with the same $\F_2$-homology of $\CP^2$ works. Even more is true: both proofs hold in the locally-flat category: the first one is easier to adapt, the second one is due to work of Lee and Wilczynski~\cite{LeeWilczynski}.

\section{Singularities and branched covers: back to dimension 2}\label{s:RiemannHurwitz}

In this section we want to go back to the branched covers we have used in the first and second proof of the adjunction formula, in Section~\ref{s:adjunction} to restrict the collection of singularities a given curve can have. Instead of giving a precise formulation (which can get a bit cumbersome), we will do two (rather similar) examples.

\begin{prop}
There exists no curve $C$ of degree $5$ with six singularities whose link is a torus knot $T(2,3)$.
\end{prop}

\begin{proof}
Suppose that such a curve existed.

The first observation we want to make is that $C$ is irreducible. As we have seen in Exercise~\ref{exe:torusknots}, singularities of type $(2,3)$ have $\beta = 1$ and $\mu = 2$. Since $\beta = 1$, each such singularity is contained in an irreducible component $C_i$ of $C$.
By the singular adjunction formula~\eqref{e:singadj}, an irreducible curve of degree $d$ has at most $\frac12(d-1)(d-2)$ singularities of type $(2,3)$. For $d = 1,2$ there cannot be any such singularity, for $d=3$ there can be at most 1, and for $d=4$ at most 3. For no non-trivial partition of $5$ can we ever reach $6$ singularities, so $C$ is irreducible.

Secondly, again by~\eqref{e:singadj}, we obtain that $C$ has no other singularities and that $\chi_g(C) = 2$, so $C$ is a rational cuspidal curve, and in particular there is a one-to-one holomorphic map $u\co \CP^1 \to \CP^2$ whose image is $C$.

Now pick one of the singularities $q \in C$ and an auxiliary line $L \subset \CP^2$, not containing $q$. As we have done in Section~\ref{s:adjunction}, project $\CP^2 \setminus \{q\}$ onto $L$, and call this projection $\pi$. Look at the composition $\pi \circ u$, which is only defined away from $u\inv(q)$. We claim (but not prove) that this function actually extends to a holomorphic function $p$ on all of $\CP^1$, so we have a branched cover $p\co \CP^1 \to \CP^1$.

We want to compute the degree of $p$: a generic line through $q$ meets $C$ at three other points, because it meets the singularity with multiplicity 2 at $q$ (convince yourself of this!). So the map $p$ has degree 3. (What happens for a non-generic line? How many of those are there?)

Whenever the line $\ell$ though $q$ passes through another singular point of $C$, there is one preimage $x\in p\inv(\ell \cap L)$ for which $e_p(x) \ge 2$.
Note that, since $p$ has degree 3, there can be at most one singularity of $C$ (other than that at $q$) on any such $\ell$.
So we have five lines, each passing through $q$ and another singular point of $C$, and five points $x_1, \dots, x_5$ in $\CP^1$ such that $e_p(x_k) > 1$. We can now use the Hurwitz formula for $p$:
\[
2 = \chi(\CP^1) = d\chi(\CP^1) - \sum_{x\in \CP^1} (e_p(x)-1) \le 3\cdot 2 - \sum_{k=1}^5 (e_p(x_k)-1) \le 1,
\]
which gives a contradiction.
\end{proof}

\begin{prop}
There exists no curve $C$ of degree $5$ with four singularities whose link is a torus knot $T(2,3)$ and one whose link is the torus knot $T(2,5)$.
\end{prop}

\begin{proof}[Sketch of proof]
Like in the previous proof, the curve $C$ is easily seen to be rational and cuspidal. This time, we project from the singularity $q_0 \in C$ of type $(2,5)$. As above, the projection induces a holomorphic map $p\co \CP^1 \to \CP^1$ of degree $3$.

The main difference with respect to the previous example is that this holomorphic map now has non-trivial ramification at the point from which we project. This is because the tangent curve to the cusp can intersect $C$ in at most one other point (this requires a small argument, which we will not give). In particular, the ramification index of the projection at the preimage $x_0 \in \CP^1$ of $q_0$ is at least 2. As above, each of the other four singularities of $C$ correspond to four ramification points $x_1, \dots, x_4$ of the projection.

Using the Hurwitz formula as in the previous proof, we obtain:
\[
2 = \chi(\CP^1) = d\chi(\CP^1) - \sum_{x\in \CP^1} (e_p(x)-1) \le 3\cdot 2 - \sum_{k=0}^4 (e_p(x_k)-1) \le 1,
\]
which again gives a contradiction.
\end{proof}

{\footnotesize
\bibliography{winterbraids}
\bibliographystyle{amsalpha}
}
\end{document}